\documentclass[a4paper,11pt,reqno]{amsart}
\usepackage{amsmath, amsthm, amsfonts, amssymb, amscd, cancel, graphicx,mathrsfs, mathtools, pinlabel, soul, stmaryrd}
\usepackage{tikz}
\usepackage{array,diagbox}
\usepackage{graphicx}
\usepackage{colortbl}
\usepackage{xcolor}
\usepackage{caption}
\usepackage{overpic}
\usepackage{thmtools}
\usepackage{thm-restate}
\usepackage[shortlabels]{enumitem}

\usepackage[T1]{fontenc}
\usepackage[utf8]{inputenc}
\usepackage{palatino}

\usepackage{xcolor}

\definecolor{amethyst}{rgb}{0.6, 0.4, 0.8} 

 \newcommand{\negcrossing}
 {\raisebox{-0.02in}
    {\includegraphics[scale=0.28, angle = 90]{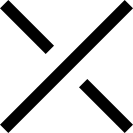}}}

 \newcommand{\poscrossing}
 {\raisebox{-0.02in}
    {\includegraphics[scale=0.28]{figures/crossing.png}}}

\newcommand{\infinity}
{\raisebox{-0.03in}
	{\includegraphics[scale=0.35]{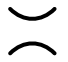}}}

\theoremstyle{definition}
\newtheorem{thm}{Theorem}
\newtheorem{defn}[thm]{Definition}
\newtheorem{prop}[thm]{Proposition}
\newtheorem{cor}[thm]{Corollary}
\newtheorem{conj}[thm]{Conjecture}
\newtheorem*{conj-nonum}{Conjecture}

\newtheorem*{quest-nonum}{Question}
\newtheorem{lemma}[thm]{Lemma}
\newtheorem{rem}[thm]{Remark}
\newtheorem{ex}[thm]{Example}
\newtheorem{question}[thm]{Question}

\newcommand{\Kh}{\mathrm{Kh}}
\newcommand{\qmax}{\mathrm{q_{max}}}
\newcommand{\qmin}{\mathrm{q_{min}}}

\usepackage{hyperref}
\usepackage{xcolor}
\hypersetup{
	colorlinks=true,
	linkcolor={magenta},
	citecolor={blue},
	urlcolor={black}}

\usepackage{cleveref}
\usepackage[backend=bibtex,maxbibnames=99,sorting=nty,style=alphabetic
]{biblatex}
\AtBeginBibliography{\footnotesize}
\DeclareFieldFormat{postnote}{#1}
\title{On the symmetric braid index of ribbon knots}
\author{Vitalijs Brejevs}
\author{Fer\.ide Ceren K\"ose}

\begin{document}
\begin{abstract}
    We define the symmetric braid index $b_s(K)$ of a ribbon knot $K$ to be the smallest index of a braid whose closure yields a symmetric union diagram of $K$, and derive a Khovanov-homological characterisation of knots with $b_s(K)$ at most three. As applications, we show that there exist knots whose symmetric braid index is strictly greater than the braid index, and deduce that every chiral slice knot with determinant one has braid index at least four. We also calculate bounds for $b_s(K)$ for prime ribbon knots with at most 11 crossings.
\end{abstract}

\maketitle

\section{Introduction and main results}
\label{sec:intro}

Introduced by Kinoshita and Terasaka in~\cite{kt:original}, the notion of a \emph{symmetric union (SU)} knot provides an appealing diagrammatic way of producing ribbon knots. Informally speaking, an SU knot $K$ is constructed from a knot $J$, called a \emph{partial knot} of $K$, by inserting twisted tangles, called \emph{twist regions}, on the symmetry axis of the connected sum $J \# -J$, where $-J$ is the mirror of $J$; see Figure~\ref{fig:su-intro} for an example and Figure ~\ref{fig:su} for the general picture. Such $K$ bounds an immersed disc all of whose singularities are of ribbon type and correspond to crossings in the chosen diagram of $J$. This construction is very versatile, as evidenced by the fact that the vast majority of ribbon knots with at most 12 crossings~\cite{lamm:nonsym} and all two-bridge slice knots~\cite{lamm:2bridge} have been shown to be SU knots, which motivates the following key open conjecture in the study of SU knots.

\begin{conj-nonum}[`Ribbon--SU conjecture']
    Every ribbon knot is an SU knot.
\end{conj-nonum}

There are multiple choices involved in the construction of an SU knot, i.e., that of a partial knot $J$, a diagram of $J$, and the way of inserting twist regions. Therefore, a natural approach to addressing the conjecture is to first choose a normal form for SU knots. One such form is provided by a result of Lamm in~\cite{lamm:original} stating that every SU knot admits a \emph{braided SU diagram}, i.e., an SU diagram that is simultaneously a diagram of the standard closure $\mathrm{cl}({\beta})$ of some braid $\beta$ that we call an \emph{SU braid}; see right of Figure~\ref{fig:su-intro} for an example and Section~\ref{sec:su} for the precise definition. Hence, we may define the \emph{symmetric braid index} $b_s(K)$ of a ribbon knot $K$ by
        \[
            b_s(K) = \min \{ b(\beta) \mid \beta \textrm{ is an SU braid s.t. } K \sim \mathrm{cl}({\beta}) \},
        \]
    where $b(\beta)$ is the index of $\beta$; if $K$ does not admit an SU diagram, we set $b_s(K) = \infty$.

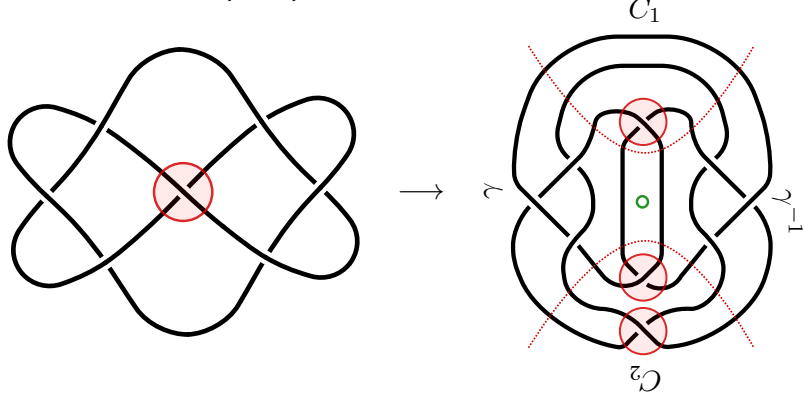
\begin{figure}[h]
        \centering
        \includegraphics[width=0.80\textwidth]{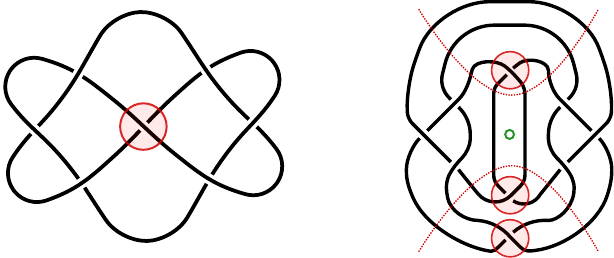}
        \begin{overpic}[width=\linewidth,height=0.001in]{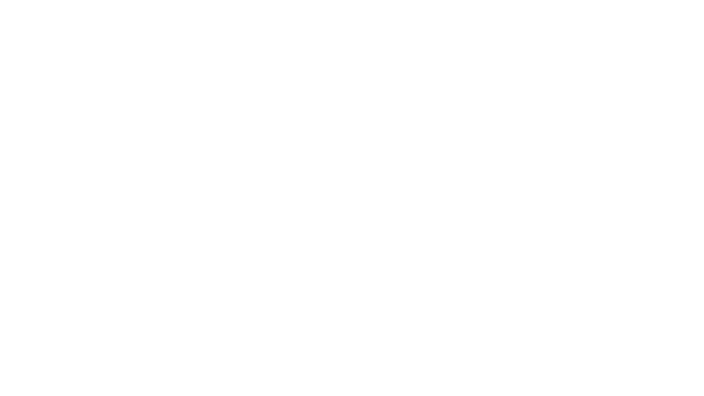}
        \put(51,20) {$\mathbf{\longrightarrow}$}
        \put(60,20) {\rotatebox{90}{$\gamma$}}
        \put(90,21) {\rotatebox{-90}{$\gamma^{-1}$}}
        \put(75,39) {$C_1$}
        \put(75,2) {\rotatebox{180}{$C_2$}}
        \end{overpic}
        \captionsetup{width=0.9\linewidth}
        \caption{Left: an SU diagram of the stevedore knot $6_1$ obtained by inserting a single twist region, shown in red, on the axis of symmetry of $3_1 \# -3_1$. Right: a braided SU diagram of $6_1$ with three twist regions and the braid axis perpendicular to the page, indicated by the green circle; the corresponding SU braid is given by $\gamma C_1 \gamma^{-1} C_2 \in B_4$ for $\gamma = \sigma_2 \sigma_1^{-1} \sigma_2$, $C_1 = \sigma_3$ and $C_2 = \sigma_1 \sigma_3^{-1}$.}        \label{fig:su-intro}
\end{figure}

Clearly, the braid index $b(K)$ of $K$ is a lower bound for $b_s(K)$. The only knot with $b_s(K) = 1$ is the unknot and no knot has $b_s(K) = 2$, so $b_s(K) = 3$ is the smallest interesting case to consider. Such knots can be represented as closures of 3-braids of the form
    \[
        \beta = \gamma \sigma_2^{\varepsilon_1} \gamma^{-1} \sigma_2^{\varepsilon_2},
    \]
where $\sigma_1$ and $\sigma_2$ are the standard generators of $B_3$, $\gamma \in B_3$ and $\varepsilon_i \in \{ \pm 1\}$ for $i = 1, 2$. Our main result is a characterisation of knots with $b_s(K) \leq 3$ in terms of their Khovanov homology.

In the following, let $K_{p,q}$ be the rational knot corresponding to $p/q \in \mathbb{Q} \cup \{1/0\}$, where $p$ is odd and positive. For $p > 1$, the knot $K_{p,q}$ is a two-bridge knot; otherwise, $K_{1,q}$ is the unknot for all $q$. Since every two-bridge knot $K_{p,q}$ also corresponds to some $p'/q' \in \mathbb{Q}$ with $p = p'$ and $0 < q' < p'$, we may assume $0 < q < p$ when working with such knots; however, elsewhere we allow rational $p/q$-tangles with $p/q \leq 1$.

\begin{restatable}{thm}{main}
{\label{thm:main}}
      Suppose that $K$ is a ribbon knot with $b_s(K) \leq 3$.
      Then there exists $K_{p,q}$ with $p^2 = \det(K)$ such that $K$ admits it as a partial knot and one of the following holds:
      \begin{enumerate}
          \item $\Kh(K) \cong \Kh(K_{p,q} \# -K_{p,q})$;
          \item $\Kh(K) \cong \Kh(K^\prime)$ for a Montesinos knot $K^\prime \in \{ K[\frac{q}{p}, \pm \frac{1}{2}, -\frac{q}{p}] \}$.
      \end{enumerate}
\end{restatable}

In~\cite{lisca:3braids}, Lisca classified finite concordance 3-braid knots into three families, denoted $(1)$, $(2)$ and $(3)$, and studied their properties such as sliceness, quasipositivity and amphichirality. Building upon results of Simone in~\cite{simone:classification}, in Section~\ref{sec:bs=3} we determine that knots with $b_s(K) \leq 3$ comprise precisely the families $(1)$ and $(2)$. This allows us to give a concise reformulation of Theorem~1.1 in~\cite{lisca:3braids}, reducing it to two cases; we defer the definitions of the notions involved in the statement to Section~\ref{sec:bs=3}.

\begin{restatable}{thm}{combinedlbs}
\label{thm:combined-lbs}
    A knot $K$ is a 3-braid knot of finite concordance order if and only if $K$ is the closure of a braid in one the following forms:
    \begin{enumerate}
        \item[$(\mathrm{A})$] $\beta = \gamma \sigma_2^{\varepsilon_1} \gamma^{-1} \sigma_2^{\varepsilon_2}$, 
        where $\gamma \in B_3$ and $\varepsilon_i \in \{\pm 1\}$ for $i = 1, 2$;
        \item[$(\mathrm{B})$] $\beta = \sigma_1^{x_1} \sigma_2^{-y_1} \dots \sigma_1^{x_n} \sigma_2^{-y_n}$,
        where the associated string of $\beta$ lies in the set
        \[
            \{ b_1 + 1, b_2, \dots, b_{k-1}, b_k + 1, c_1, \dots, c_l \mid k + l \geq 2 \}
        \]
        for linear dual strings $(b_1, \dots, b_k)$ and $(c_1, \dots, c_l)$.
    \end{enumerate}
    Moreover, in case $(\mathrm{A})$, the knot $K$ is always ribbon and it is quasipositive when $\varepsilon_1 = \varepsilon_2 = 1$, and in case $(\mathrm{B})$, it is amphichiral.
\end{restatable}


Lisca also showed in~\cite{lisca:3braids} that family $(1)$ is disjoint from the other two, but both complements $ (2) \setminus (3)$ and $(3) \setminus (2)$ are non-empty; however, the only known knots in the latter set are non-slice. Using Theorem~\ref{thm:main}, we can obstruct certain ribbon 3-braid knots from having $b_s(K) \leq 3$, which implies that they belong to $(3)$, but not to $(1) \cup (2)$. This provides first known examples of ribbon knots with distinct symmetric and regular braid indices.


\begin{restatable}{thm}{bs}
    {\label{thm:bs>b}}
    For  $K \in \{ 10_{99}, 10_{123} \}$, we have $3 = b(K) < b_s(K) \in \{4,5\}$.
\end{restatable}

The proofs of Theorems~\ref{thm:main} and~\ref{thm:bs>b} rely on the work~\cite{kose:knotinv} of the second author who determined that symmetric fusion number one knots (see Definition~\ref{def:sufusion}) are precisely Montesinos knots $ K[\frac{q}{p}, \pm \frac{1}{n}, -\frac{q}{p}]$ and gave an explicit formula for their Khovanov homology, presented in Lemma~\ref{lemma:khformula}. The formula also leads to a concise Khovanov-theoretic obstruction to a knot being a finite concordance order 3-braid knot, proved in Section~\ref{sec:kh}. In the following, for a link $L$ we write
    \begin{align*}
        \qmax(L) &= \max \{ j \mid \Kh^{i,j}(L) \neq 0\} \ \textrm{ and }  \\
        \qmin(L) &= \min \{ j \mid \Kh^{i,j}(L) \neq 0\}.
    \end{align*}
\begin{cor}
\label{cor:qmax}
Suppose $K$ is a finite concordance order 3-braid knot.
  \begin{enumerate}
           \item If $K$ is in the family $(1)$, then $\qmax(K)+\qmin(K) = \pm 8$;
           \item otherwise, $\qmax(K)+\qmin(K) = 0$.
    \end{enumerate}
\end{cor}

Corollary~\ref{cor:qmax} gives another way to show that family $(1)$ is disjoint from $(2) \cup (3)$. It also provides a method for obstructing slice knots from having (regular) braid index 3, as illustrated by the following example.

\begin{ex}
    Let $K$ be the stevedore knot $6_1$. Because $K$ is not a torus knot and can be expressed as a 4-braid closure (see Figure~\ref{fig:su-intro}), we have $b(K) \in \{ 3,4 \}$. However, $\qmax(K) + \qmin(K) = 4$, so by Corollary~\ref{cor:qmax} we conclude that $b(K) = 4$.\footnote{We note that this a well-known result that also follows, e.g., from the MFW inequality.}
\end{ex}


Moreover, if the knot $K$ in the statement of Theorem~\ref{thm:main} has $\det(K)=1$, it follows that its Khovanov homology is isomorphic to that of $K_{1,q} \# -K_{1,q}$ or $K[\frac{q}{1},\pm\frac{1}{2},-\frac{q}{1}]$, which represent the unknot. Since Khovanov homology detects the unknot~\cite{kronheimermrowka:unknot}, we combine Theorem~\ref{thm:main} with the fact that knots in Lisca's family $(3)$ are amphichiral to get the following corollary. 

\begin{cor}{\label{cor:slice-obs}}
    Let $K$ be a chiral knot with $\det(K)=1$ and $b(K) \leq 3$. Then $K$ has infinite concordance order; in particular, $K$ is not slice.
\end{cor}

Corollary~\ref{cor:slice-obs} can also give lower bounds on the braid index of slice knots.

\begin{ex}
    The Kinoshita--Terasaka family $\{K_n\}_{n \in \mathbb{Z} \setminus \{0\}}$ of one-twist SU knots from~\cite{kt:original}, shown in Figure~\ref{fig:kt}, is chiral by~\cite{kose:amphichiral}. Furthermore, $\det(K_n) = 1$ for any $n$ since $K_n$ admits the unknot as a partial knot; see Theorem~\ref{thm:det}. Then, by Corollary~\ref{cor:slice-obs}, we have $b(K_n) \geq 4$ for any $n$.
\end{ex}

\begin{figure}[h]
        \centering
        \includegraphics[width=0.35\textwidth]{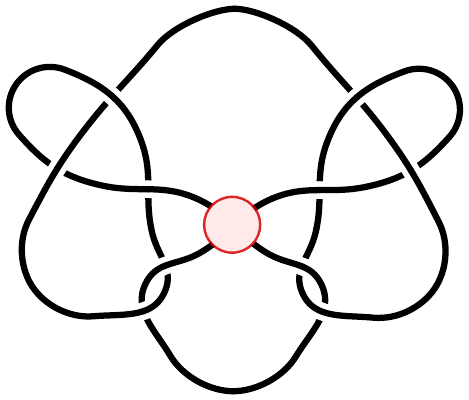}
        \begin{overpic}[width=\linewidth,height=0.001in]{figures/blank.png}
        \put(49.1,16.2) {$\tau$}
        \end{overpic}
        \captionsetup{width=0.9\linewidth}
        \caption{The Kinoshita--Terasaka family $\{K_n\}_{n \in \mathbb{Z} \setminus \{0\}}$ is obtained by inserting $n$-tangles in the twist region $\tau$.}
        \vspace{-1em}
        \label{fig:kt}
\end{figure}

\subsection*{Structure of the paper} In Section~\ref{sec:su} we define SU knots and SU braids, and in Proposition~\ref{prop:2bridge} prove a crucial fact about two-bridge partial knots of an SU knot with $b_s(K) \leq 4$. In Section~\ref{sec:bs=3} (resp., Section~\ref{sec:kh}) we prove Theorem~\ref{thm:combined-lbs} (resp., Theorems~\ref{thm:main} and~\ref{thm:bs>b}). Section~\ref{sec:search} contains data on symmetric braid indices of SU knots with at most 11 crossings, while Section~\ref{sec:future} proposes some open questions.

\subsection*{Acknowledgements} VB was supported by the Austrian Science Fund grant `Cut and Paste Methods in Low Dimensional Topology' and gratefully acknowledges the hospitality of the University of Georgia, where this project was conceived.

\section{Symmetric unions and their braidings}
\label{sec:su}

Let $D_J \subset \mathbb{R}^2$ be an unoriented diagram of a knot $J$. After choosing a straight line axis $A \subset \mathbb{R}^2$ with $D_J \cap A = \varnothing$, reflect $D_J$ about $A$ to get a diagram $-D_J$ of the mirror image of $J$. Then, for some $k \geq 0$, choose $k+1$ discs $T_0, T_1, \dots, T_k$ embedded disjointly in $\mathbb{R}^2$ such that each $T_i$ is invariant under reflection about $A$ and intersects the diagram $D_J \sqcup -D_J$ in two unknotted arcs. Fix an integer $\mu$ such that $1 \leq \mu \leq k+1$. Viewing each $T_i$ as a rational 0-tangle, replace it by an $m_i$-tangle $T_i'$, where $m_i = \infty$ for $0 \leq i \leq \mu-1$ and $m_i \in \mathbb{Z} \setminus \{0\}$ for $\mu \leq i \leq k$. Choose an orientation on the resulting diagram and denote it by $D_L = (D_J \sqcup -D_J)(\infty_\mu, n_1, \dots, n_l)$, where $l = k - \mu + 1$ and $n_i = m_{\mu -1 + i}$, and call the tangles $T_{\mu}', \dots, T_k'$ \emph{twist regions}. The construction is illustrated in Figure~\ref{fig:su}.

\begin{figure}[h]
\centering
\begin{overpic}
    [scale=.18]{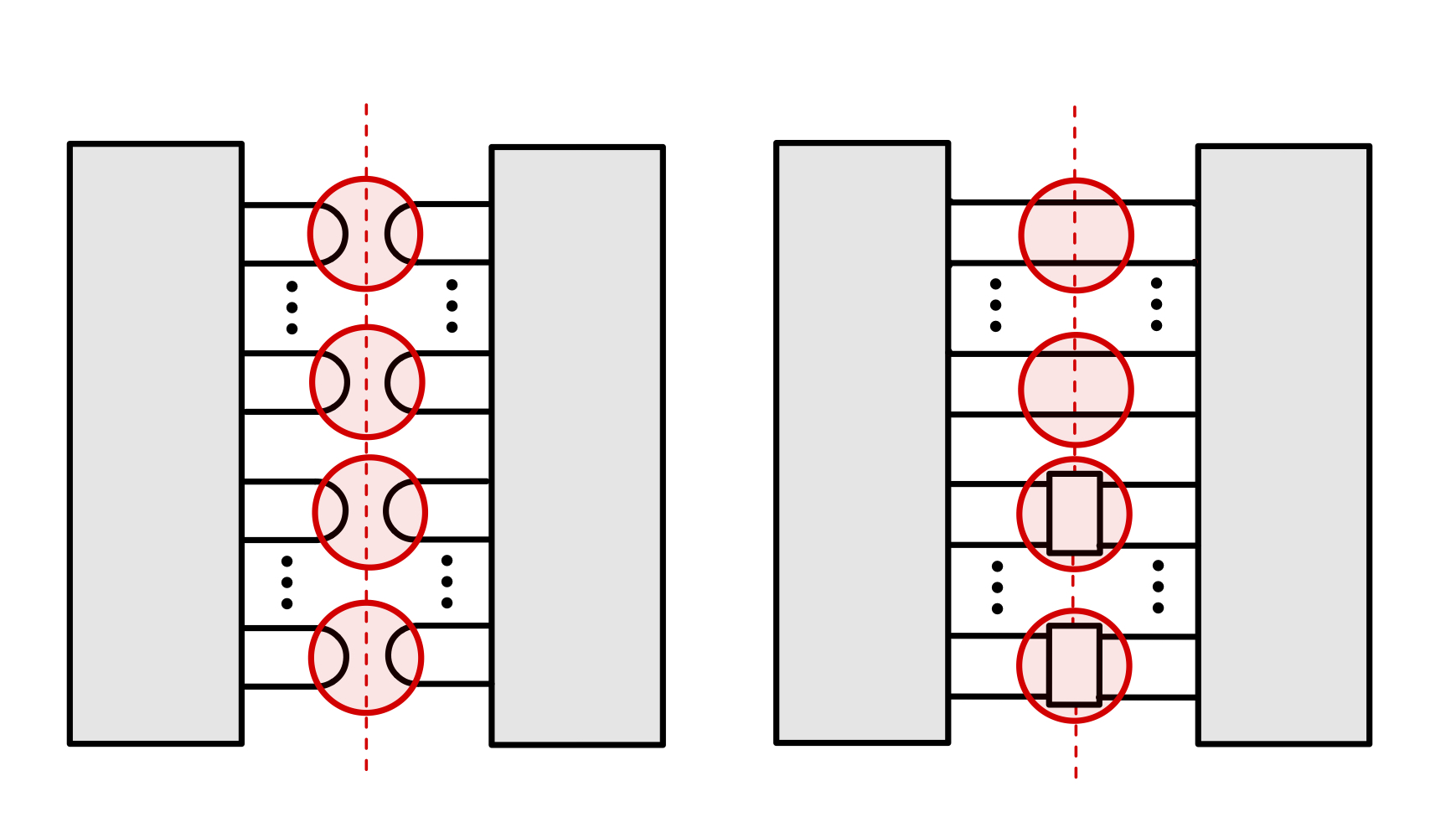}
    \put(21,46){$\scriptstyle T_0$}
    \put(21,35.8){$\scriptstyle T_{\mu-1}$}
    \put(20.5,26.4){$\scriptstyle T_{\mu}$}
    \put(21,16.5){$\scriptstyle T_{k}$}
    \put(70,46){$\scriptstyle T_0'$}
    \put(70,35.6){$\scriptstyle T_{\mu-1}'$}
    \put(68.5,26.1){$\scriptstyle T_{\mu}'$}
    \put(69.5,16.4){$\scriptstyle T_{k}'$}
    \put(72.5,21.7){$\scriptstyle n_1$}
    \put(72.5,11.4){$\scriptstyle n_l$}
    \put(8.5,25){$D_J$}
    \put(36.5,25){$-D_J$}
    \put(57,25){$D_J$}
    \put(85,25){$-D_J$}
    \put(26,47){$A$}
    \put(75,47){$A$}
    \put(47,25){$\longrightarrow$}
\end{overpic}
\vspace{-1em}
\caption{Construction of the symmetric union diagram $(D_J \sqcup -D_J)(\infty_\mu,\allowbreak n_1, \dots, n_l)$; rectangles labelled $n_1, \dots, n_l$ contain the respective number of `vertical' half-twists.}\label{fig:su}
\end{figure}
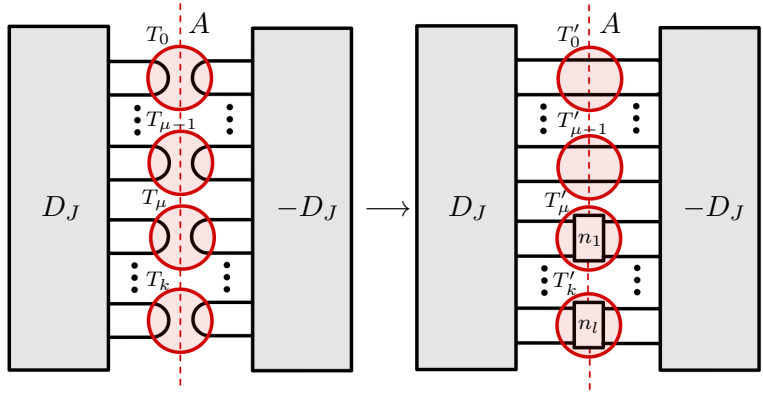

\begin{defn}
\label{def:sulink}
    A link $L$ is a \emph{symmetric union (SU) link} if it admits a diagram $D_L$ as above for some $J$, $D_J$, $\mu$, $l$ and $n_i$ for $1 \leq i \leq l$. The knot $J$ is called a \emph{partial knot} of $L$.
\end{defn}

Observe (cf.~\cite[Remark~2.2]{lamm:original}) that an SU link $L$ has $\mu$ components, hence $L$ is a knot if and only if $\mu = 1$; in the following, we shall primarily be concerned with this case.

In general, an SU knot $K$ can admit many different partial knots; however, the determinant of any partial knot is prescribed by that of $K$.

\begin{thm}[{\cite[Theorem~2.6]{lamm:original}}]
    \label{thm:det}
    If $K$ is an SU knot, then for any partial knot $J$ of $K$ we have $\det(J)^2 = \det(K)$.
\end{thm}

By~\cite[Theorem~4.3]{lamm:original}, a link $L$ is an SU link if and only if it can be represented as the closure of an \emph{SU braid} $\beta$, i.e., a braid given by
    \begin{equation*}
    \label{eq:beta}
    \tag{$\dagger$}
        \beta = \gamma C_1 \gamma^{-1} C_2 \in B_n
    \end{equation*}
for some $n \geq 1 $, where $\gamma \in B_n$ is arbitrary and $C_i$ are given by the following:
    \begin{enumerate}
        \item if $n = 1$, then $C_1$ and $C_2$ are trivial;
        \item if $n = 2$, then $C_1 = \textrm{id} \in B_2$ and $C_2 \in \{ \sigma_1^{\pm 1} \}$;
        \item if $n \geq 3$, then
            \begin{align*}
                C_1 &= \begin{cases}
                    \sigma_2^{\varepsilon_{12}} \sigma_4^{\varepsilon_{14}} \dots \sigma_{n-1}^{\varepsilon_{1(n-1)}} & \textrm{if $n$ is odd} \\
                    \sigma_3^{\varepsilon_{13}} \sigma_5^{\varepsilon_{15}} \dots \sigma_{n-1}^{\varepsilon_{1(n-1)}} & \textrm{if $n$ is even, and}
                \end{cases}\\
                C_2 &= \begin{cases}
                    \sigma_2^{\varepsilon_{22}} \sigma_4^{\varepsilon_{24}} \dots \sigma_{n-1}^{\varepsilon_{2(n-1)}} & \textrm{if $n$ is odd} \\
                    \sigma_1^{\varepsilon_{21}} \sigma_3^{\varepsilon_{23}} \dots \sigma_{n-1}^{\varepsilon_{2(n-1)}} & \textrm{if $n$ is even}
                \end{cases}
            \end{align*}
        such that $\varepsilon_{ij} \in \{\pm 1\}$ for all $i$ and $j$.
    \end{enumerate}
    
\begin{defn}
    A diagram $D_L$ of a link $L$ is a \emph{braided SU diagram} if it is the standard diagram of the closure of an SU braid $\beta \in B_n$. The \emph{symmetric braid index} $b_s(L)$ of $L$ is the minimal $n \geq 1$ for which $L$ admits a braided SU diagram, with $b_s(L) = \infty$ if $L$ is not an SU link.
\end{defn}

Fix $m \geq 1$ and suppose that $K$ is an SU knot with $b_s(K) \in \{ 2m-1, 2m \}$. Then $K$ can be constructed by taking $D_J$ to be the plat closure of a braid $\gamma \in B_{2m}$. We illustrate the case $m = 2$ in Figure~\ref{fig:plat}, where the diagram $D_J$ is a 4-plat closure, and hence, provided $\det(K) \neq 1$, the partial knot $J$ is a two-bridge knot~$K_{p,q}$ for some $0 < q < p$ (see, e.g.,~\cite[Section~12.B]{bzh:knots}). Since $\det(K_{p,q}) = p$ for any $p$, by Theorem~\ref{thm:det} the following result holds.

\begin{prop}
\label{prop:2bridge}
    If $K$ has $b_s(K) \in \{3, 4\}$ and $\det(K) \neq 1$, then $K$ must admit a partial knot $J$ belonging to the finite set
        \[
            \mathcal{J}_p = \{ K_{p,q} : 0 < q < p \text{ and } p^2 = \det(K) \}.
        \]
    On the other hand, if $\det(K)=1$, then $K_{p,q}$ is necessarily the unknot $ K_{1,q}$.
\end{prop}

\begin{figure}[h!]
\centering
\begin{overpic}
    [scale=.16]{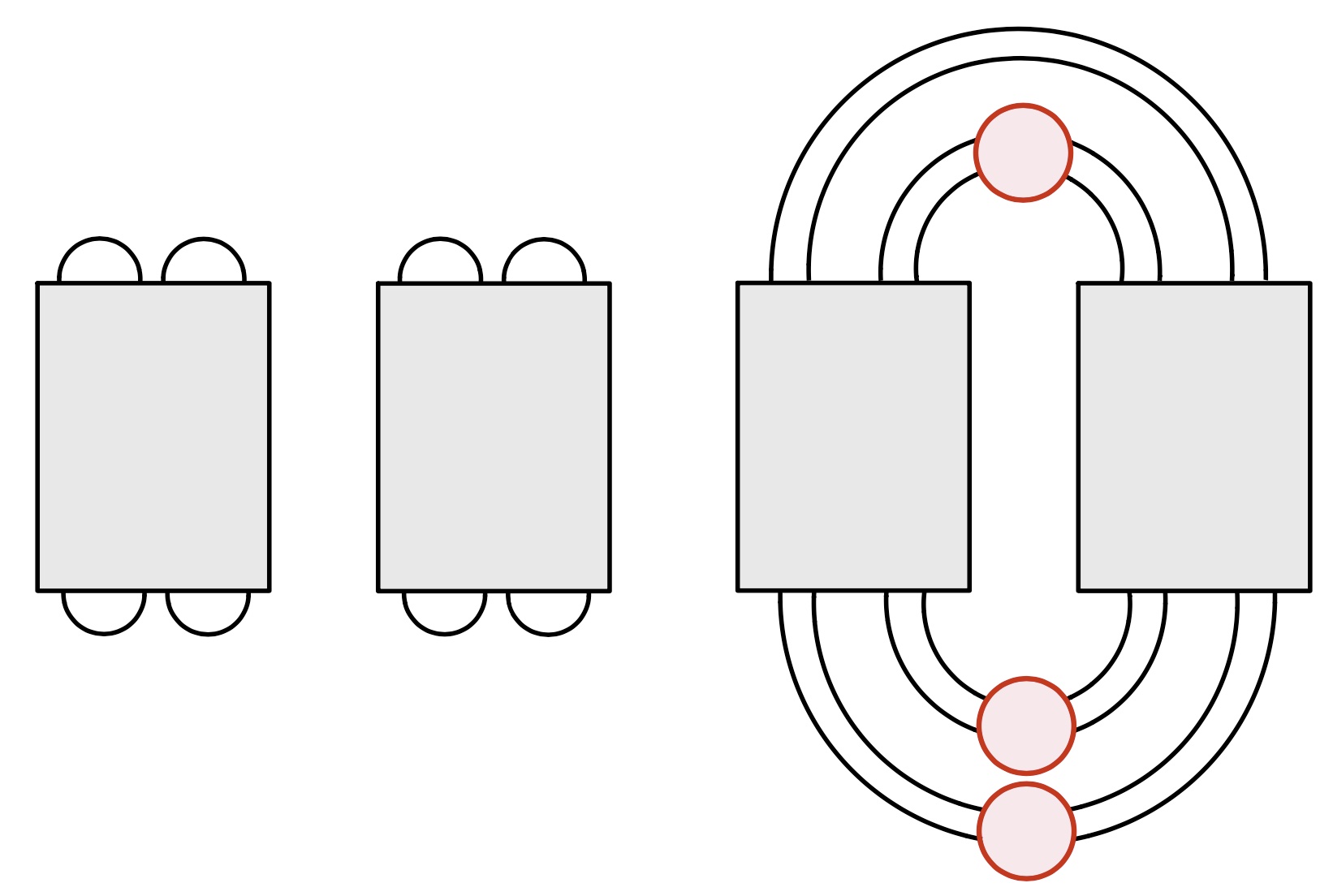}
    \put(10,14) {$J$}
    \put(33,14) {$-J$}
    \put(10,33) {\rotatebox{90}{$\gamma$}}
    \put(35,36) {\rotatebox{-90}{$\gamma^{-1}$}}
    \put(46.5,33) {$\longleftrightarrow$}
    \put(63,33) {\rotatebox{90}{$\gamma$}}
    \put(88,36) {\rotatebox{-90}{$\gamma^{-1}$}}
    \put(75.5,55){\scalebox{0.9}{$\tau_1$}}
    \put(75.5,12){\scalebox{0.9}{$\tau_2$}}
    \put(75.5,4){\scalebox{0.9}{$\tau_3$}}
    \put(60,4) {$K$}
\end{overpic}
\captionsetup{width=0.9\linewidth}
\caption{A knot $K$ with $b_s(K) \in \{ 3, 4 \}$ admits an SU diagram with the partial knot $J$ given by the plat closure of $\gamma \in B_4$, where each twist region $\tau_i$ contains one crossing. If $\gamma$ does not contain any $\sigma_1$ letters, then the topmost free strand in the right diagram can be moved all the way down so that the $\tau_3$ region can be removed via the Reidemeister I move; this yields the closure of an SU 3-braid. (Cf.~\cite[Figure~13]{lamm:original}.)}
\label{fig:plat}
\end{figure}

\section{Classification of knots with symmetric braid index 3}
\label{sec:bs=3}

First let us discuss the classification of finite concordance order 3-braid knots proved by Lisca in~\cite{lisca:3braids} and briefly introduced in Section~\ref{sec:intro}. The classification consists of three families $(1)$, $(2)$ and $(3)$; the following is an outline of their main properties:
\begin{itemize}
    \item family $(1)$ consists of knots $K$ such that $K$ or $-K$ is the closure of a braid of the form
        \[
        \label{eq:fam1}
        \tag{$*$}
            \beta = (\sigma_1 \sigma_2)^3 \sigma_1^{x_1} \sigma_2^{-y_1} \dots \sigma_1^{x_n} \sigma_2^{-y_n},
        \]
    where $n, x_i, y_i \geq 1$, $\sum_i y_i = \sum_i x_i + 4$, and all $x_i$ and $y_i$ are subject to a technical combinatorial condition; moreover, $\beta$ is quasipositive and $K$ is ribbon;
    \item family $(2)$ consists of braided SU knots obtained as closures of braids of the form
        \[
            \beta = \gamma \sigma_2^{\pm 1} \gamma^{-1} \sigma_2^{\mp 1},
        \]
    where $\gamma \in B_3$; in particular, $(2)$ contains the unknot;
    \item family $(3)$ consists of certain alternating amphichiral knots with dihedral symmetry obtained as closures of braids of the form
        \[
            \beta = \sigma_1^{x_1} \sigma_2^{-y_1} \dots \sigma_1^{x_n} \sigma_2^{-y_n},
        \]
    where $n \geq 2$ and $x_i, y_i \geq 1$.
\end{itemize}

Family $(1)$ is disjoint from $(2) \cup (3)$, but families $(2)$ and $(3)$ have non-trivial intersection. According to the KnotInfo database~\cite{knotinfo}, there are knots in $(2)$, such as $10_{48}$ and $12_{a1011}$, that are chiral, hence $ (2) \setminus (3)$ is non-empty. On the other hand, $(3) \setminus (2)$ contains all non-slice finite concordance order 3-braid knots, such as the figure-eight knot. In Theorem~\ref{thm:combined-lbs} we will show that, in fact, families $(1)$ and $(2)$ consist precisely of knots with $b_s(K) \leq 3$, whilst Theorem~\ref{thm:bs>b} exhibits ribbon knots in $(3) \setminus (2)$.

In order to state and prove Theorem~\ref{thm:combined-lbs}, we now recall some relevant terminology regarding 3-braids, established in~\cite{lisca:3braids}. For any $\beta \in B_3$ of the form
    \[
        \beta = (\sigma_1 \sigma_2)^{3t} \sigma_1^{x_1} \sigma_2^{-y_1} \dots \sigma_1^{x_n} \sigma_2^{-y_n}
    \]
with $n, x_i, y_i \geq 1$, $t \in \mathbb{Z}$, we define the \emph{associated string} $\mathbf{a}(\beta)$ of $\beta$ to be the string 
    \[
        (2^{[x_1-1]}, y_1 + 2, 2^{[x_2-1]}, y_2 + 2, \dots, 2^{[x_n - 1]}, y_n + 2),
    \]
where $2^{[m]}$ denotes the integer $2$ repeated $m$ times. The pair $(t, \mathbf{a}(\beta))$ uniquely recovers $\beta$ and determines the isotopy type of $\mathrm{cl}(\beta)$ up to cyclic rotation and reflection of $\mathbf{a}(\beta)$.

The \emph{linear dual} of a string $\mathbf{b} = (b_1, \dots, b_k)$ with all $b_i \geq 2$ is defined as follows:
    \begin{itemize}
        \item if $\mathbf{b} = (2^{[m]})$, then its linear dual is the string $\mathbf{c} = (m+1)$;
        \item if there exists $j$ such that $b_j \geq 3$, write $\mathbf{b}$ as
            $$
                \mathbf{b} = (2^{[m_1]}, 3 + n_1, 2^{[m_2]},\allowbreak 3 + n_2, \dots, 2^{[m_l]}, 2 + n_l)
            $$ with all $m_i, n_i \geq 0$;
        then its linear dual is the string
            $$
                \mathbf{c} = (2 + m_1, 2^{[n_1]},\allowbreak 3 + m_2, 2^{[n_2]}, 3 + m_3, \dots, 3 + m_l, 2^{[n_l]}).
            $$
    \end{itemize}
One can verify that the notion of a linear dual is symmetric. Now we prove the following.

\combinedlbs*

\begin{proof}[Proof of Theorem~\ref{thm:combined-lbs}]

First, we show that Lisca's families $(1)$ and $(2)$ are included in case $(\mathrm{A})$. In view of the discussion at the start of the section, it is, in fact, sufficient to show that knots in family $(1)$ are closures of 3-braids of the form $\gamma \sigma_2^{\pm 1} \gamma^{-1} \sigma_2^{\pm 1}$. To this end, we leverage previous work~\cite{simone:classification,brejevs-simone:chain} by Simone and the first author and Simone; it contains a classification of all links arising as closures of 3-braids of the form
\[
    \beta = (\sigma_1 \sigma_2)^{3t} \sigma_1^{x_1} \sigma_2^{-y_1} \dots \sigma_1^{x_n} \sigma_2^{-y_n}
\]
with $n, x_i, y_i \geq 1$ and $t \in \{ \pm 1 \}$, that are \emph{$\chi$-slice}, which is a generalisation of the usual notion of knot sliceness to links. In particular, this classification includes all knots in family $(1)$. It follows from Theorem~1.11 in~\cite{brejevs-simone:chain} that, up to mirroring, 
knots in family $(1)$ are closures of 3-braids with $t = -1$ whose associated strings lie in the family 
\[
    \mathcal{S}_{1a} = \{ (b_1, \dots, b_k, 2, c_l, \dots, c_1, 2 \mid k+l \geq 3 \},
\]
where $(b_1, \dots, b_k)$ and $(c_1, \dots, c_l)$ are linear duals. Furthermore, Section~2 in~\cite{brejevs:ribbon} and Figure~14 in~\cite{brejevs-simone:chain} imply that such knots emerge as closures of 3-braids of the form
    \[
        \beta' = \Delta^{-2} \sigma_1 \overline{\zeta} \sigma_1^2 \zeta^{-1} \sigma_1,
    \]
where $\zeta \in B_3$ is arbitrary, $\Delta = \sigma_1 \sigma_2 \sigma_1 = \sigma_2 \sigma_1 \sigma_2$ is the Garside element in $B_3$, and $\overline{\,\cdot\,} : B_3 \rightarrow B_3$ is the map defined by $\sigma_1 \mapsto \sigma_2$ and $\sigma_2 \mapsto \sigma_1$. Recalling that $\Delta^{-1} \sigma_1 = \sigma_2 \Delta^{-1}$ and $\sigma_2 \Delta^{-1} = \Delta^{-1} \sigma_1$, we can write $\Delta^{-1} \sigma_1 \overline{\zeta} = \sigma_2 \zeta \Delta^{-1}$, and hence
    \begin{align*}
        \beta' &= \Delta^{-1} \sigma_2 \zeta \Delta^{-1} \sigma_1^2 \zeta^{-1} \sigma_1 \\
        &= (\sigma_2^{-1} \sigma_1^{-1} \sigma_2^{-1}) \sigma_2 \zeta (\sigma_1^{-1} \sigma_2^{-1} \sigma_1^{-1}) \sigma_1^2 \zeta^{-1} \sigma_1 \\
        &= \sigma_2^{-1} \sigma_1^{-1} \zeta \sigma_1^{-1} \sigma_2^{-1} \sigma_1 \zeta^{-1} \sigma_1 \\
        & = \sigma_2^{-1} \gamma \sigma_2^{-1} \gamma^{-1} \\
        &\sim \gamma \sigma_2^{-1} \gamma^{-1} \sigma_2^{-1},
    \end{align*}
where $\gamma = \sigma_1^{-1} \zeta \sigma_1^{-1}$ and $\sim$ denotes conjugacy in $B_3$. Since closures of conjugate braids are isotopic, it follows that knots in family $(1)$, up to mirroring, are SU knots with symmetric braid index 3 and corresponding SU braids given by
    \[
        \beta'' = \gamma \sigma_2^{\pm 1} \gamma^{-1} \sigma_2^{\pm 1}.
    \]

Now we show that Lisca's family $(3)$ is included in case $(\mathrm{B})$. Comparing the definition of family $(3)$ in~\cite{lisca:3braids} with Section~1 in~\cite{simone:classification} (particularly with Example~1.3 and the discussion after Corollary~1.11), one can see that knots in family $(3)$ correspond to closures of alternating 3-braids whose associated strings belong to the family
    \[
        \mathcal{S}_{2c} = \{ b_1 + 1, b_2, \dots, b_{k-1}, b_k + 1, c_1, \dots, c_l \mid k + l \geq 2 \}.
    \]
This yields case $(\mathrm{B})$.

Clearly, in case $(\mathrm{A})$ the knot $K$ is ribbon as it is an SU knot. The statements about quasipositivity and amphichirality also follow immediately from Theorem~1.1 in~\cite{lisca:3braids}, which concludes the proof.
\end{proof}

\section{Khovanov homology of knots with symmetric braid index 3}
\label{sec:kh}

Having understood the placement of SU knots with $b_s(K) \leq 3$ among all finite concordance order 3-braid closures, we now derive their Khovanov-homological characterisation. Let us first define an auxiliary class of SU knots to be used in the following, studied previously in~\cite{moore:su}.

\begin{defn}
\label{def:sufusion}
    Let $K_n$ denote the SU knot $(D_J \sqcup -D_J)(\infty_1, n)$ for some partial knot $J$. If the SU knot $(D_J \sqcup -D_J)(\infty_2)$ obtained by replacing the axial $n$-tangle by the $\infty$-tangle is the two-component unlink, say that $K_n$ has \emph{symmetric fusion number one}.
\end{defn}

Such knots $K_n$ are shown to be exactly the Montesinos knots $ K[\frac{q}{p}, \frac{1}{n},- \frac{q}{p}]$ and their Khovanov homology is computed in~\cite{kose:knotinv}. This generalises \cite[Lemma 10]{moore:su}. 

\begin{thm}[\cite{kose:knotinv}]
\label{thm:kose}
Let $K_n$ be an SU knot that admits a symmetric fusion number one diagram $(D_J \sqcup -D_J)(\infty_1, n)$. Then $K_n$ is a Montesinos knot $K[\frac{q}{p},\frac{1}{n},-\frac{q}{p}]$ and $J = K_{p,q}$ is either the unknot or a two-bridge knot. 
\end{thm}

Recall that the Khovanov homology of a link $L$ can be compactly described by its \textit{Khovanov polynomial} $\Kh(L)(t,q)$, given by

\begin{align*}
    \Kh(L)(t,q) = \sum_{i,j} \mathrm{rank}\,(\Kh^{i,j}(L))\,t^iq^j.
\end{align*}

\begin{lemma}[\cite{kose:knotinv}]
\label{lemma:khformula}
Let $K_n$ denote the Montesinos knot $ K[\frac{r}{p}, \frac{1}{n},- \frac{r}{p}]$ and let $$ V(K_{p,r} \# -K_{p,r}) = \sum_{i=-m}^{m} a_i q^{2i}$$ such that $a_m \neq 0$ be the Jones polynomial of $K_{p,r} \# -K_{p,r}$.
Then the Khovanov homology of $K_n$ is determined by $n$ and the coefficients of $V(K_{p,r} \# -K_{p,r})$ as follows:

\begin{align*}
    \Kh(K_n)(t,q)  = 
\begin{cases}
q^{-1} + q & \text{ if $p=1$,} \\
    q^{-1}  \left( 1 +q^2 + (1+tq^4)  t^n q^{2n}  \sum_{k=-m}^{m-1} b_k t^k q^{2k}   \right) & \text{otherwise},
\end{cases}    
\end{align*}

\noindent
where $b_k = b_{-k-1}$ and $b_k = (-1)^{k+1}\sum_{i=k+1}^{m} a_i$ for $k \in \{0,1, \dots , m-1\}$.  
\end{lemma}

\begin{rem}\hfill
    \begin{enumerate}[(a)]
        \item For $p>1$, it is guaranteed that $m \geq 1$ since $V(K_{p,r}) \neq 1$ for any two-bridge knot $K_{p,r}$. Thus the formula for $\Kh(K)(t,q)$ in Lemma~\ref{lemma:khformula} is well-defined.
        \item By Proposition~\ref{prop:2bridge}, there are only finitely many possible Khovanov homologies for an SU knot with $b_s(K)=3$ and fixed determinant.
    \end{enumerate}
\end{rem}


The SU knots with symmetric braid index 3 appear in~\cite{watson:identicalkh} as a subset of a family of knots $K_\gamma(T,U)$, where $\gamma \in B_3$, and $T$ and $U$ are tangles. The tangle $T^\sigma$ is obtained by adding a half-twist to $T$, as shown in Figure~\ref{fig:TU-knots}, and the tangle $T^{\overline{\sigma}}$ is obtained by adding the opposite half-twist to $T$ in the same way.

We say that the tangle pair $(T,U)$ is \emph{compatible} if the writhe $w(K_\gamma(T,U)) = w(K_\gamma(T^\sigma, U^{\overline{\sigma}}))$; the notion of compatibility is independent of $\gamma$ and is purely diagrammatic since writhe is not a knot invariant. The sum $T + U$ of two tangles is defined by stacking $T$ on top of $U$. This generalises the half-twist action: $T^\sigma$ (resp., $T^{\overline{\sigma}}$) may be denoted $T + \poscrossing$ (resp., $T + \negcrossing$). A tangle $T$ is called \emph{simple} if $T + \infinity$ is isotopic (fixing the endpoints) to the tangle $\infinity$ . Note that $T^\sigma$ is simple if and only if $T$ is simple.

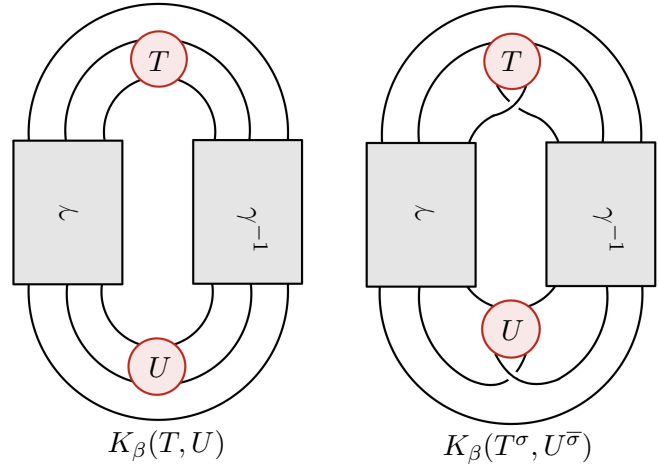
\begin{figure}[h]
\centering
\begin{overpic}
    [scale=.12]{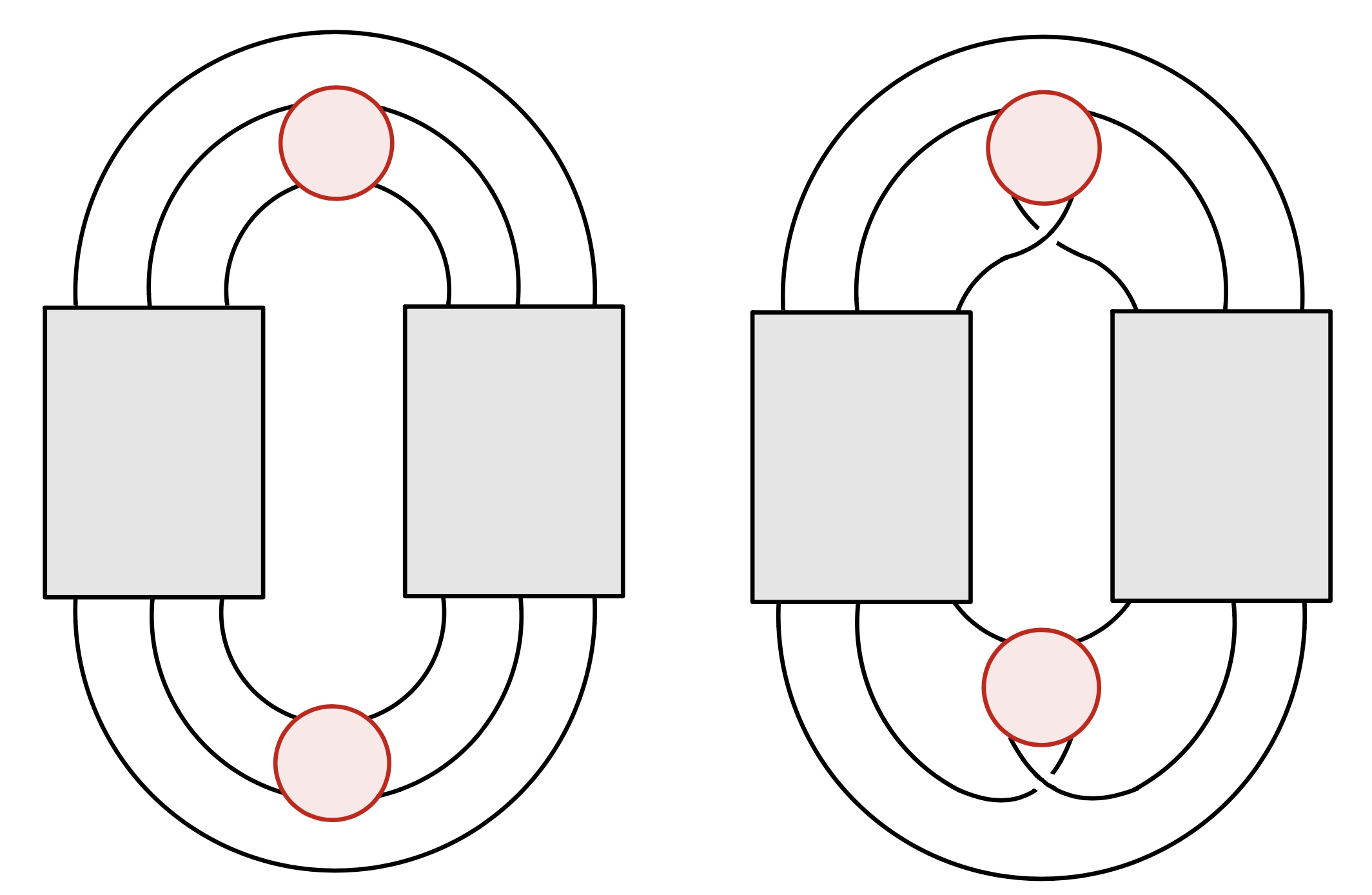}
    \put(10,31) {\rotatebox{90}{$\gamma$}}
    \put(36,33) {\rotatebox{-90}{$\gamma^{-1}$}}
    \put(62,31) {\rotatebox{90}{$\gamma$}}
    \put(88,33) {\rotatebox{-90}{$\gamma^{-1}$}}
    \put(23,53) {$T$}
    \put(23,8) {$U$}
    \put(74.5,53) {$T$}
    \put(74.5,13.5) {$U$}
    \put(17,-3) {$K_\gamma(T, U)$}
    \put(66,-3.5) {$K_\gamma(T^\sigma, U^{\overline{\sigma}})$}
\end{overpic}
\vspace{1em}
\captionsetup{width=0.9\linewidth}
\caption{The knots $K_\gamma(T,U)$ on the left and $K_\gamma(T^\sigma, U^{\overline{\sigma}})$ on the right defined in ~\cite{watson:identicalkh}. Here when $T$ and $U$ are the rational tangles $\frac{p}{q}$ for $p=1$, $K_\gamma(T,U)$ is an SU knot. In particular, if $T$ and $U$ are $(\pm 1)$-tangles, it is an SU knot with $b_s=3$.}\label{fig:TU-knots}
\end{figure}


\begin{thm}[\cite{watson:identicalkh}]\label{thm:watson}
    Suppose that $(T,U)$ is a compatible pair of simple tangles. Then 
    $$ \Kh(K_\gamma(T,U)) \cong \Kh(K_\gamma(T^\sigma, U^{\overline{\sigma}})).$$
\end{thm}

We are now equipped for the proof of our main result.

\main*

\begin{proof}[Proof of Theorem~\ref{thm:main}]
If $b_s(K) < 3$, then $K$ is the unknot and the result follows immediately, hence suppose $b_s(K) = 3$. The statement about $K$ admitting a two-bridge partial knot $K_{p,q}$ follows from Proposition~\ref{prop:2bridge}. 


Up to mirroring, we may assume that $K$ is represented as the closure of a braid $\gamma \sigma_2\gamma^{-1}\sigma_2^{\pm 1}$ for some $\gamma \in B_3$. Observe that the corresponding braided SU diagram of $K$ is a special case of the diagram in Figure \ref{fig:TU-knots}, where the tangle $T$ is the $(-1)$-tangle and the tangle $U$ is a $\mp 1$-tangle, using the conventions in Section~\ref{sec:su}. It is easy to check that the pair $(T,U)$ is compatible and that each tangle is simple. Then, by Theorem \ref{thm:watson} we have $\Kh(K) \cong \Kh(K'_{\mp})$, where $K'_{\mp} = (D_{K_{p,q}} \sqcup -D_{K_{p,q}})(\infty_1, \mp 1 + 1)$.

Now observe that 
$K'_{\mp}$ has symmetric fusion number one since replacing the only twist region of $K'_{\mp}$ by the $\infty$-tangle allows one to cancel all crossings in $\gamma$ and $\gamma^{-1}$, which produces a two-component unlink. By Theorem~\ref{thm:kose}, we have that $K'_{\mp}$ is $K[\frac{q}{p}, \frac{1}{n},- \frac{q}{p}]$, where $n = \mp 1 + 1 $. Observe that for $K[\frac{q}{p}, \frac{1}{n},- \frac{q}{p}]$, its mirror $- K[\frac{q}{p}, \frac{1}{n},- \frac{q}{p}] = K[\frac{q}{p}, -\frac{1}{n},- \frac{q}{p}]$ for any $n$. Hence, both $K'_-$ and its mirror are isotopic to $ K[\frac{q}{p}, \frac{1}{0},- \frac{q}{p}] = K_{p,q} \# -K_{p,q}$, while $K'_+$ is isotopic to at least one of $ K[\frac{q}{p}, \pm \frac{1}{2},- \frac{q}{p}]$, which yields the result.\qedhere


\end{proof}

Theorem~\ref{thm:bs>b} now follows by combining Theorem~\ref{thm:main} and Proposition~\ref{prop:2bridge}.


\begin{proof}[Proof of Theorem~\ref{thm:bs>b}]
    Let $K = 10_{99}$ and suppose $b_s(K) = 3$. Since $\det(K) = 81$, by Proposition~\ref{prop:2bridge}, $K$ admits an element of $\mathcal{J}_9 = \{ K_{9,1}, K_{9,4} \} = \{ 9_1, 6_1 \}$ as a partial knot. The Jones polynomial of $K_{9,1} \# -K_{9,1}$ is given by
        \begin{align*}
            V(K_{9,1} \# -K_{9,1}) &= q^{-18} + {q^{-16}} -2{q^{-14}} + 3{q^{-12}} - 4{q^{-10}}\\
            &+ 5{q^{-8}} - 6{q^{-6}} + 7{q^{-4}} - 7{q^{-2}} + 9\\
            &- 7q^{2} + 7q^{4} - 6q^{6} + 5q^{8} - 4q^{10}\\
            &+ 3q^{12} - 2q^{14} + q^{16} - q^{18},
        \end{align*}
    hence by Lemma~\ref{lemma:khformula} we have $\qmax(K_{9,1} \# -K_{9,1})=19$ and $$\qmax(K[\tfrac{1}{9},\tfrac{1}{2},-\tfrac{1}{9}]) = -\qmin(K[\tfrac{1}{9},-\tfrac{1}{2},-\tfrac{1}{9}]) = 23.$$ Since $\qmax(K) = -\qmin(K) = 11$, we have that $K_{9,1}$ is not a partial knot for any braided SU diagram of $K$ on three strands. Similarly, we compute that $\qmax(K_{9,4} \# -K_{9,4}) = 13$ and $$\qmax(K[\tfrac{4}{9},\tfrac{1}{2},-\tfrac{4}{9}]) = -\qmin(K[\tfrac{4}{9},-\tfrac{1}{2},-\tfrac{4}{9}]) = 17,$$ which yields a contradiction.
    
    The proof for $K = 10_{123}$ is analogous, with $\det(K) = 121$ and $\mathcal{J}_{11} = \{ K_{11,1}, K_{11,3}, K_{11,5} \} = \{ 11_{a367}, 6_2, 7_2 \}$. The statement that $b_s(K) \in \{ 4,5 \}$ for both knots follows from the existence of index five SU braids, presented in Table~\ref{tab:bs}, whose closures yield SU diagrams for $10_{99}$ and $10_{123}$.
\end{proof}

\begin{proof}[Proof of Corollary~\ref{cor:qmax}] The result is immediate if $K$ is in family $(3)$ as the knots in family $(3)$ are amphichiral. If $K$ is in family $(1)$ or family $(2)$, by Theorem~\ref{thm:main} and Lemma~\ref{lemma:khformula}, it follows that
\begin{align*}
    \qmax(K) &= \max \{1, 2n+2m-1\} \text{ and} \\
    \qmin(K) &= \min \{-1, 2n-2m+1\}
\end{align*}
since, in the notation of Lemma~\ref{lemma:khformula}, we have $b_{m-1}=b_{-m}=(-1)^ma_m \neq 0$, where $a_m$ is the leading coefficient of $V(K_{p,r} \# -K_{p,r})$ and $n =\pm 2$ or $0$.

Let $\mathrm{br}(K)$ denote the \emph{breadth} of the Jones polynomial of $K$ defined by
     $$\mathrm{br}(K) = M(K) - m(K),$$
where $M(K)$ and $m(K)$ are the largest and the smallest powers of $q$ in $V(K)$, respectively. In general it holds that $\mathrm{br}(K) \leq 2c(K)$ for any $K$, where $c(K)$ is the crossing number of $K$, but, in particular, $\mathrm{br}(K) = 2c(K)$ if $K$ is alternating~\cite{kauffman:jones,murasugi:jones,thistlethwaite:jones}.
     
Two-bridge knots are alternating, and so are their connected sums. Thus, $m \geq 3$ since
\begin{align*}
        4m = \mathrm{br}(K_{p,r} \# -K_{p,r})  = 2 c(K_{p,r} \# -K_{p,r}) \geq 2 c( K_{3,1} \# -K_{3,1} )= 12.
    \end{align*}

Since $n =\pm 2$ or $0$ and $m \geq 3$, it follows that $\qmax(K) = 2n+2m-1$ and $\qmin(K)=2n-2m+1$, hence $\qmax(K)+\qmin(K) = 4n$. This completes the proof as $n= \pm 2$ if $K$ is in family $(1)$ and $n=0$ if $K$ is in family $(2)$. \qedhere
\end{proof}

\section{Symmetric braid indices of small ribbon knots}
\label{sec:search}

In Table~\ref{tab:bs} we present our computations of symmetric braid indices for the 50 ribbon knots with at most 11 crossings. Among these, there are six which do not have known SU diagrams; among the rest, we could not exactly determine the symmetric braid index for 10 knots, whilst for the remaining 34 examples, symmetric and regular braid indices coincide. For most of the knots in the table, we found the corresponding SU braids $\gamma C_1 \gamma^{-1} C_2$ by iterating over candidate $\gamma$ up to length 12 and index six, and all possible combinations of $C_1$ and $C_2$, using \textsc{SnapPy}~\cite{SnapPy} to identify the isotopy type of the closure, up to mirroring; the knot $11_{a326}$, however, we braided by hand using the diagram from~\cite{lamm:nonsym}.

There are eight known SU knots in Table~\ref{tab:bs} that are not obstructed from having $b_s(K) = 4$, yet for which we were unable to find the corresponding SU 4-braids via the initial approach. For these knots, we deepened our search by applying the fact that each 4-plat diagram of a rational knot $K_{p,q}$ is determined by a continued fraction expansion of $p/q$, where we allow positive and negative coefficients within the same expansion, which correspond to positive and negative crossings, respectively. This enabled us to efficiently generate more candidate $\gamma$ of index 4 for these eight knots, arising from values of $p/q$ of all knots $K_{p,q} \in \mathcal{J}_p$ for appropriate $p$. However, we were unable to find SU 4-braids for any of the knots in question using this procedure, which serves as evidence that their symmetric braid indices may be strictly greater than four.

\begin{conj}
\label{conj:bs>4}
    If $K$ is one of the SU knots $10_{99}$, $10_{123}$, $10_{140}$, $11_{a164}$, $11_{a316}$, $11_{a326}$, $11_{n42}$ and $11_{n172}$, then $b_s(K) > 4$.
\end{conj}

\section{Some future directions}
\label{sec:future}

In order to resolve Conjecture~\ref{conj:bs>4}, one could seek a Khovanov-theoretic description of knots with $b_s(K) = 4$, particularly since they enjoy many of the properties of knots with $b_s(K) = 3$: each knot with symmetric braid index 4 must admit a partial knot from a finite set and has symmetric fusion number one, as replacing the $\tau_1$ twist region in Figure~\ref{fig:plat} by the $\infty$-tangle produces the two-component unlink. However, we do not have a classification of symmetric fusion number one knots with two twist regions, and there does not exist an analogue of Theorem~\ref{thm:watson} applicable to SU 4-braid closures. Hence, we leave the following question open.

\begin{question}
    Does there exist a characterisation of knots with $b_s(K) = 4$ similar to Theorem~\ref{thm:main}? For knots with $b_s(K) = n$ for $n \geq 5$?
\end{question}

In addition, observe that symmetric and regular braid indices agree for the two-bridge knots in Table~\ref{tab:bs}, namely $6_1$, $8_8$, $8_9$, $9_{27}$, $10_3$, $10_{22}$, $10_{35}$, $10_{42}$ and $11_{a96}$. Since ribbon two-bridge knots are SU knots by~\cite{lamm:2bridge}, one may wonder if this apparent coincidence reflects a general fact.

\begin{question}
    If $K$ is a ribbon two-bridge knot, then does $b(K) = b_s(K)$?
\end{question}

\footnotesize
\begin{table}[h]
\centering

\begin{tabular}{c||c|c||c|c|c}
$K$  & $b(K)$ & $b_s(K)$ & $\gamma$ & $C_1$ & $C_2$            \\
\hline
$6_1$  & $4$        & $4$           & $2\underline{1}2$  & $3$ & $1\underline{3}$ \\
$8_8$  & $4$        & $4$           & $2\underline{1} 2 2$ & $3$ & $1\underline{3}$ \\
$8_9$  & $3$        & $3$           & $1 1\underline{2} 1$ & $2$ & $\underline{2}$ \\
$8_{20}$ & $3$        & $3$           & $1 1 1$ & $2$ & $2$ \\
$9_{27}$ & $4$        & $4$           & $2\underline{1} \underline{3} 2 2$ & $3$ & $1 \underline{3}$     \\
$9_{41}$ & $5$        & $5$           & $\underline{3} 2 1 2 4 3 3$ & $2 4$ & $2 \underline{4}$ \\
$9_{46}$ & $4$        & $4$           & $2\underline{1} 2$ & $3$ & $1 3$        \\
$10_3$ & $6$ & $6$ & $2 4\underline{1} 3 3 5 2 4$  & $3 5$ & $13\underline{5}$ \\
$10_{22}$ & $4$ & $4$ & $2\underline{1}222$ & $3$ & $1\underline{3}$ \\
$10_{35}$ & $6$ & $6$ & $2\underline{1} \underline{4} 3 2 \underline{4}$ & $35$ & $\underline{1}3\underline{5}$ \\
$10_{42}$ & $5$ & $5$ & $3 3 \underline{4}\underline{2} 1\underline{2} 3$ & $24$ & $\underline{2}\underline{4}$ \\
$10_{48}$ & $3$ & $3$ & $111\underline{2}1$ & $2$ & $\underline{2}$ \\
$10_{75}$ & $5$ & $[5,6]$ & $2\underline{1} \underline{4} 3 2 5 \underline{4}$ & $35$ & $\underline{1}3\underline{5}$ \\
$10_{87}$ & $4$ & $4$ & $2\underline{1}\underline{1} \underline{3} 2 2$ & $3$ & $1\underline{3}$ \\
\rowcolor{lightgray!50}$10_{99}$ & $3$ & $[4,5]$ & $1 1 1 3 \underline{4} 3$ & $24$ & $\underline{2}\underline{4}$ \\
\rowcolor{lightgray!50}$10_{123}$ & $3$ & $[4,5]$ & $1 3\underline{2} 1 \underline{4} 3$ & $24$ & $\underline{2}\underline{4}$ \\
$10_{129}$ & $4$ & $4$ & $2\underline{1}22$ & $3$ & $\underline{1}3$ \\
$10_{137}$ & $5$ & $5$ & $1 2 3\underline{2} 1 3$ & $24$ & $2\underline{4}$ \\
$10_{140}$ & $4$ & $[4,5]$ & $1113$ & $24$ & $24$ \\
$10_{153}$ & $4$ & $4$ & $21112$ & $3$ & $1\underline{3}$ \\
$10_{155}$ & $3$ & $3$ & $11\underline{2}1$ & $2$ & $2$ \\
$11_{a28}$ & $4$ & $4$ & $2\underline{1} \underline{3} 2 \underline{3} 2$ & $3$ & $1\underline{3}$ \\
$11_{a35}$ & $4$ & $4$ & $2\underline{3311}22$ & $3$ & $1\underline{3}$ \\
$11_{a36}$ & $5$ & $5$ & $1234 2 4 3 3 4\underline{1} 4\underline{1} \underline{3} $ & $2\underline{4}$ & $\underline{2}4$ \\
$11_{a58}$ & $5$ & $5$ & $3\underline{142114}23\underline{1}424321 $ & $2\underline{4}$ & $\underline{2}4$ \\
$11_{a87}$ & $5$ & $5$ & $1234\underline{2} \underline{4} 1 3\underline{2} 4 3 3$ & $2 \underline{4}$ & $2 \underline{4}$ \\
$11_{a96}$ & $6$ & $6$ & $2\underline{1} 4 \underline{3} 2 4$ & $35$ & $1\underline{3}\underline{5}$ \\
\rowcolor{red!20}$11_{a103}$ & $6$ & $[6, \infty]$ & --- & --- & --- \\
$11_{a115}$ & $5$ & $5$ & $1234\underline{2}\underline{2} \underline{4} 1 3 2 \underline{3}$ & $2 \underline{4}$ & $\underline{2} \underline{4}$ \\
$11_{a164}$ & $4$ & $[4,5]$ & $1234\underline{2} 4 4 3 3 4 1 2 1 \underline{3}$ & $2\underline{4}$ & $\underline{2}4$ \\
\rowcolor{red!20}$11_{a165}$ & $5$ & $[5, \infty]$ & --- & --- & --- \\
$11_{a169}$ & $5$ & $5$ & $1234\underline{2}\underline{2}\underline{4}\underline{3}21413$ & $2 \underline{4}$ & $\underline{2} 4$ \\
\rowcolor{red!20}$11_{a201}$ & $6$ & $[6, \infty]$ & --- & --- & --- \\
$11_{a316}$ & $4$ & $[4,5]$ & $1234 2 4 4\underline{1} 2 4 \underline{3} 2$ & $2\underline{4}$ & $\underline{2}4$ \\
$11_{a326}$ & $4$ & $[4,7]$ & $5 3 1 6 \underline{4} 2 5 \underline{3}\underline{1} 5 5 3 3 \underline{1}$ & $2\underline{4}\underline{6}$ & $\underline{2}\underline{4}6$ \\
$11_{n4}$ & $4$ & $4$ & $2\underline{1} \underline{3} 2 2$ & $3$ & $\underline{1}3$ \\
$11_{n21}$ & $4$ & $4$ & $2\underline{1}\underline{1} 2 2$ & $3$ & $1\underline{3}$ \\
$11_{n37}$ & $4$ & $4$ & $2\underline{1} \underline{1}\underline{1} 2$ & $3$ & $1\underline{3}$ \\
$11_{n39}$ & $4$ & $4$ & $2\underline{1}\underline{1} \underline{3} 2$ & $3$ & $\underline{1}\underline{3}$ \\
$11_{n42}$ & $4$ & $[4,5]$ & $1 1\underline{2} 3 2 \underline{1}$ & $24$ & $\underline{2}4$ \\
$11_{n49}$ & $5$ & $5$ & $1 1 3\underline{2}\underline{1} 3$ & $24$ & $\underline{2}\underline{4}$ \\
$11_{n50}$ & $4$ & $4$ & $2 2\underline{1} 2$ & $3$ & $13$ \\
\rowcolor{red!20}$11_{n67}$ & $5$ & $[5, \infty]$ & --- & --- & --- \\
\rowcolor{red!20}$11_{n73}$ & $4$ & $[4, \infty]$ & --- & --- & --- \\
\rowcolor{red!20}$11_{n74}$ & $4$ & $[4, \infty]$ & --- & --- & --- \\
$11_{n83}$ & $5$ & $5$ & $3144\underline{3}2\underline{3}1$ & $24$ & $\underline{2}\underline{4}$ \\
$11_{n116}$ & $5$ & $5$ & $1 1 3\underline{2}\underline{1} 3$ & $24$ & $2\underline{4}$ \\
$11_{n132}$ & $4$ & $4$ & $2\underline{1} 2 2$ & $3$ & $13$ \\
$11_{n139}$ & $5$ & $[5,6]$ & $2\underline{1}24$ & $35$ & $135$ \\
$11_{n172}$ & $4$ & $[4,6]$ & $\underline{2} 4 1 1 \underline{3} 5\underline{2} 4$ & $35$ & $\underline{1}35$ \\

\end{tabular}
\vspace{0.5\baselineskip}
\captionsetup{width=\linewidth}
\caption{Bounds for $b_s(K)$ for ribbon knots with at most $11$ crossings, up to taking the mirror image. Corresponding SU braids are presented in the form~\eqref{eq:beta}, with $i$ and $\underline{i}$ denoting $\sigma_i$ and $\sigma_i^{-1}$, respectively. Knots that do not possess known SU diagrams are highlighted in red, and knots with distinct symmetric and regular braid indices in grey.}
\label{tab:bs}
\end{table}

\printbibliography
\end{document}